\documentclass[12pt,reqno]{amsart}
\usepackage{amsmath}
\usepackage{amsthm}
\usepackage{amsfonts}
\usepackage{amssymb}
\usepackage{mathrsfs}
\usepackage[margin=1.35in, top =1.25in, bottom = 1.25in]{geometry}
\usepackage{setspace}
\usepackage{tikz}
\usetikzlibrary{matrix, arrows}
\usepackage{tikz-cd}
\usepackage{enumitem}
\usepackage[colorlinks, breaklinks]{hyperref}
\hypersetup{linkcolor=black,citecolor=blue,filecolor=black,urlcolor=black} 
\usepackage{pdflscape}
\usepackage{array}
\usepackage{adjustbox}
\usepackage{tabularx}
\usepackage{multirow}
\usepackage{multicol}
\usepackage{float}

\usepackage{color}

\usepackage{verbatim}
\usepackage{fancyvrb}
\usepackage{dsfont}
\usepackage[ruled,linesnumbered,vlined]{algorithm2e}

\newcommand{\ceil}[1]{\lceil #1 \rceil}
\DeclareMathOperator{\crit}{crit}

\renewcommand{\epsilon}{\varepsilon}
\renewcommand{\bar}{\overline}
\newcommand{\FF}{\mathbb{F}}
\newcommand{\floor}[1]{\lfloor #1 \rfloor}
\newcommand{\kk}{K}					
\newcommand{\norm}[1]{\left\Vert #1 \right\Vert}
\renewcommand{\phi}{\varphi}
\DeclareMathOperator{\lce}{lce}
\newcommand{\m}{\mathfrak{m}}			
\newcommand{\mult}{\mathcal{J}}			
\newcommand{\NN}{\mathbb{N}}
\newcommand{\ones}{\mathbf{1}}			
\newcommand{\RR}{\mathbb{R}}
\renewcommand{\SS}{\mathcal{S}}

\DeclareMathOperator{\trunc}{trunc}

\renewcommand{\vec}[1]{\mathbf{#1}}
\newcommand{\ZZ}{\mathbb{Z}}

\DeclareMathOperator{\ch}{char}

\newtheorem{thm}{Theorem}[section]
\newtheorem{lemma}[thm]{Lemma}
\newtheorem{prop}[thm]{Proposition}
\newtheorem{cor}[thm]{Corollary}

\newtheorem{proposition}[thm]{Proposition}
\newtheorem{corollary}[thm]{Corollary}

\theoremstyle{definition}
\newtheorem{defn}[thm]{Definition}
\newtheorem{algo}[thm]{Algorithm}
\newtheorem{observation}[thm]{Observation}
\newtheorem{example}[thm]{Example}

\newtheorem{sharp:exercise}[exercise]{${}^\sharp$Exercise}
\newtheorem{rmk}[thm]{Remark}
\newtheorem{question}[thm]{Question}
\newtheorem{convention}[thm]{Convention}
\newtheorem{notation}[thm]{Notation}
\newtheorem{definition}[thm]{Definition}
\newtheorem{remark}[thm]{Remark}

\numberwithin{equation}{section}
\numberwithin{figure}{section}

\author[C.A.~Francisco]{Christopher A. Francisco}
\address{Department of Mathematics, Oklahoma State University,
401 Mathematical Sciences, Stillwater, OK 74078}
\email{chris.francisco@okstate.edu}
\urladdr{\url{https://math.okstate.edu/people/chris/}}
\author[M.~Mastroeni]{Matthew Mastroeni}
\address{Department of Mathematics, Oklahoma State University,
401 Mathematical Sciences, Stillwater, OK 74078}
\email{mmastro@okstate.edu}
\urladdr{\url{https://mnmastro.github.io/}}
\author[J.~Mermin]{Jeffrey Mermin}
\address{Department of Mathematics, Oklahoma State University,
401 Mathematical Sciences, Stillwater, OK 74078}
\email{mermin@math.okstate.edu}     
\urladdr{\url{https://math.okstate.edu/people/mermin/}}   
\author[J.~Schweig]{Jay Schweig}
\address{Department of Mathematics, Oklahoma State University,
401 Mathematical Sciences, Stillwater, OK 74078}
\email{jay.schweig@okstate.edu}
\urladdr{\url{https://math.okstate.edu/people/jayjs/}}

\begin{document}

\title{
Computing Generalized Frobenius Powers of Monomial Ideals}
\date{}

\maketitle

\begin{abstract}
Generalized Frobenius powers of an ideal were introduced in \cite{Frobenius:powers} as characteristic-dependent analogs of test ideals. However, little is known about the Frobenius powers and critical exponents of specific ideals, even in the monomial case. We describe an algorithm to compute the critical exponents of monomial ideals, and use this algorithm to prove some results about their Frobenius powers and critical exponents. Rather than using test ideals, our algorithm uses techniques from linear optimization. 
\end{abstract}

\begin{spacing}{1.1}

\section{Introduction}

\paragraph{} Frobenius powers of an ideal with non-negative real-valued exponents were introduced in \cite{Frobenius:powers} as a characteristic-dependent analog of test ideals in F-finite regular domains of prime characteristic.  The motivation for defining Frobenius powers was to find a prime characteristic invariant sensitive enough to mimic a property of multiplier ideals in characteristic zero, namely that the multiplier ideal $\mult(I^\lambda)$ agrees with the multiplier ideal $\mult(f^\lambda)$ for a general $f \in I$.  In particular, it is known that $\mult(f^\lambda) = \mult(I^\lambda)$ when $I$ is the monomial ideal generated by the terms of $f$ \cite{suffgen}. Thus computations for arbitrary ideals can be reduced to the monomial case.

Let $S = \kk[x_1, \dots, x_m]$ be a standard graded polynomial ring over a field $\kk$ of characteristic $p > 0$.  We let $F: S \to S$ denote the Frobenius homomorphism of $S$.  Recall that a ring $S$ of characteristic $p$ is called \emph{F-finite} if $S$ is a finitely generated as a module over the subring $F(S) = S^p = \{ f^p \mid f \in S\}$.  We will assume throughout that $\kk$ is an F-finite field so that the polynomial ring $S$ is also F-finite.

The main result of our paper is Algorithm~\ref{algo:second}, a deterministic algorithm which computes all the critical exponents, and hence all the fractional Frobenius powers, of an arbitrary monomial ideal over K.  This algorithm does not involve test elements; instead, it uses analytic geometry, base $p$ arithmetic, and a generalization of long division.  The algorithm appears to be very efficient in characteristic 2 and 3, and slower in large characteristic.  Immediate corollaries to the algorithm are that all critical exponents are rational, and that the set of critical exponents is closed under multiplication by $p$.

Section~\ref{section:background} recalls necessary background on Frobenius powers and arithmetic in base $p$, and introduces notation that will be used throughout the paper, including the $(p,d)$-Sierpinski simplex, a fractal that we use to describe the Frobenius powers.
Section~\ref{section:algo} describes the algorithm for computing the critical exponents of a monomial ideal.  The algorithm is worked out in some detail for the ideal $(x^{2}y^{2},y^{3}z^{3})$ over $\mathbb{F}_{3}$ and $\mathbb{F}_{5}$ in the appendix. 

In Section~\ref{section:examples}, we use our techniques to compute some Frobenius powers and critical exponents in more generality.  Proposition~\ref{prop:squarecube} does this computation for $(x^{2}y^{2},y^{3}z^{3})$ in all characteristics simultaneously, demonstrating as a proof-of-concept that this is possible.  Proposition~\ref{prop:heightone} computes the least critical exponent for any height one monomial ideal containing a pure power.  We close with some questions for further research and an appendix that demonstrates Algorithm~\ref{algo:first}.

\section{Background on Frobenius Powers}\label{section:background}

\subsection{Frobenius Powers}

\paragraph{} Let $S$ be an F-finite standard graded polynomial ring as above.  In \cite{Frobenius:powers}, Hern\'andez-Teixeira-Witt define the \emph{Frobenius powers} of an ideal $I \subseteq S$ as a family of ideals $I^{[\lambda]}$ parametrized by a non-negative real number $\lambda$, which agree with the usual Frobenius powers $I^{[p^e]} = (f^{p^e} \mid f \in I)$ when $\lambda = p^e$.  As one might hope, this family of ideals has good containment properties:

\begin{prop}[{\cite[3.16]{Frobenius:powers}}]
Let $I, J \subseteq S$ be ideals, and let $\lambda, \mu \in \RR_{\geq 0}$.  Then:
\begin{enumerate}[label = \textnormal{(\alph*)}]
\item \textnormal{(Monotonicity)} If $\lambda < \mu$, then $I^{[\lambda]} \supseteq I^{[\mu]}$.
\item \textnormal{(Right Constancy)}  For every $\lambda$, there exists an $\epsilon > 0$ such that $I^{[\mu]} = I^{[\lambda]}$ whenever $\lambda \leq \mu < \lambda + \epsilon$.
\item $I^{[\lambda]}I^{[\mu]} \supseteq I^{[\lambda + \mu]}$
\item For any ideal $J \subseteq S$, we have $I^{[\lambda]}J^{[\lambda]} \supseteq (IJ)^{[\lambda]}$.
\end{enumerate}
\end{prop}

\vspace{1 ex}

Similar to jumping numbers and F-jumping numbers of multiplier ideals and test ideals, we are interested in determining the Frobenius powers of various monomial ideals $I$ and the real numbers $\lambda > 0$ such that $I^{[\mu]} \neq I^{[\lambda]}$ for all $\mu < \lambda$, which are called the \emph{critical exponents} of $I$.  In particular, there is a smallest positive critical exponent by right constancy; it is called the \emph{least critical exponent} of $I$ and is denoted by $\lce(I)$.

In the remainder of this subsection, we summarize the stages by which generalized Frobenius powers are constructed, and we make some simple observations that simplify the case of monomial ideals to working over $\FF_p$.  Given an ideal $I \subseteq S$ and $\lambda \in \RR_{\geq 0}$, the Frobenius power $I^{[\lambda]}$ is constructed as follows:
\begin{itemize}
\item If $\lambda = k$ is an integer with base $p$ expansion $k = k_0 + k_1p + \cdots k_rp^r$, then 
\[ I^{[k]} = I^{k_0}(I^{k_1})^{[p]} \cdots (I^{k_r})^{[p^r]} \]
\item If $\lambda = \frac{k}{q} \in \ZZ[\frac{1}{p}]_{\geq 0}$ is a non-negative $p$-adic rational, we define $I^{[\frac{k}{q}]} = (I^{[k]})^{[\frac{1}{q}]}$, where for any ideal $J$, the ideal $J^{[1/q]}$ is the smallest ideal $L$ such that $L^{[q]} \supseteq J$ as originally defined in \cite{eth:roots}.  In practice, because we are ultimately interested in ideals $J$ in a polynomial ring over $\FF_p$, the ideal $J^{[\frac{1}{q}]}$ is always easily computable by \cite[2.5]{eth:roots}.
\item For any real number $\lambda \geq 0$, the Frobenius power $I^{[\lambda]}$ is then defined by taking any monotone decreasing sequence $(\lambda_j)$ of $p$-adic rationals converging to $\lambda$ from above.  The monotonicity of Frobenius powers then yields an ascending chain of ideals $I^{[\lambda_1]} \subseteq I^{[\lambda_2]} \subseteq \cdots$, and $I^{[\lambda]}$ is defined to be the stable value of this chain, which exists since $S$ is Noetherian.  
\end{itemize}
In particular, every real Frobenius power is the Frobenius power of some $p$-adic rational.

\begin{prop} \label{prop:twopart}
Let $\phi: S \to T$ be a ring homomorphism between F-finite regular domains, $I \subseteq S$ be an ideal, and $\lambda \in \RR_{\geq 0}$.  Then:
\begin{enumerate}[label = \textnormal{(\alph*)}]
\item $(IT)^{[\lambda]} \subseteq I^{[\lambda]}T$, with equality if $\lambda$ is an integer.
\item If in addition $S$ is free over $S^q$ with basis $e_1, \dots, e_n$, $T$ is free over $T^q$, and $\phi(e_1), \dots, \phi(e_n)$ are part of a basis for $T$ over $T^q$, then $(IT)^{[\frac{k}{q}]} = I^{[\frac{k}{q}]}T$.\end{enumerate}
\end{prop} 

\begin{proof}
(a) If $\lambda = k$ is an integer, it is clear that $(IT)^{[k]} = I^{[k]}T$ since homomorphisms preserve both ordinary powers and $p$-th power Frobenius powers.
As $(I^{[\frac{k}{q}]})^{[q]} \supseteq I^{[k]}$, we have $(I^{[\frac{k}{q}]}T)^{[q]} = (I^{[\frac{k}{q}]})^{[q]}T \supseteq I^{[k]}T = (IT)^{[k]}$ so that $ I^{[\frac{k}{q}]}T \supseteq (IT)^{[\frac{k}{q}]}$.  The claim then follows for arbitrary $\lambda$ by applying the preceding inclusions to a monotone decreasing sequence of $p$-adic rationals converging to $\lambda$.

(b) By the previous part, it is enough to show that $I^{[\frac{1}{q}]}T = (IT)^{[\frac{1}{q}]}$ for any ideal $I \subseteq S$.  For $f \in I$, write $f = \sum_i f_i^qe_i$, so $\phi(f) = \sum_i \phi(f_i)^q\phi(e_i)$. By \cite[2.5]{eth:roots}, both the ideals $(IT)^{[\frac{1}{q}]}$ and $I^{[\frac{1}{q}]}T$  are generated by all elements of the form $\phi(f_i)$ for some $f \in I$.
\end{proof}

\begin{rmk}
It is worth noting two important cases in which one can apply the second part of Proposition~\ref{prop:twopart}:

\begin{itemize}
\item $T$ is obtained from $S$ by extension of the ground field.  This reduces computations of the Frobenius powers of monomial ideals to computations over $\FF_p$.

\item $S = \FF_p[x_1, \dots, x_m]$, $T = \FF_p[y_1, \dots, y_s]$, and $\phi(x_1), \dots, \phi(x_m)$ are square-free monomials with disjoint supports. In this case, for any $\vec{x}^\vec{a} = x_1^{a_1} \cdots x_m^{a_m}$ with $a_i < q$ for all $i$, we have $\phi(\vec{x}^\vec{a}) = \vec{y}^\vec{b}$ for some $\vec{b} \in \NN^m$ with $b_j < q$ for all $j$ by assumption so that $\phi(\vec{x}^\vec{a})$ is part of the monomial basis for $T$ over $T^q$.
\end{itemize}
\end{rmk}

\subsection{Frobenius Powers of Monomial Ideals}

\paragraph{}  In this subsection, we fix the notation used throughout the rest of the paper and make some simple observations about the Frobenius powers of monomial ideals.

\begin{notation}\label{n:vectors}
If $\vec{x}^\vec{b} = x_{1}^{b_{1}}\dots x_{m}^{b_{m}}$ is a monomial of $S$, we say that $\vec{b} = (b_1, \dots, b_m) \in \NN^m$ is the \emph{exponent vector} of $\vec{x}^\vec{b}$.  Let $I = (\vec{x}^{\vec{a}_1}, \dots, \vec{x}^{\vec{a}_n})$ be a proper monomial ideal in $S$, and let $A = (\vec{a}_1 | \cdots | \vec{a}_n)$ be the $m \times n$ matrix whose columns are the exponent vectors of the generating monomials of $I$.  For any $\vec{u}=(u_1,\dots, u_n) \in \NN^n$ and $k \in \NN$, we set $\norm{\vec{u}} = \sum_i u_i$, and we recall that the \emph{multinomial coefficient} $\binom{k}{\vec{u}}$ is equal to $\frac{k!}{u_1! u_2!\cdots u_n!}$ if $\norm{\vec{u}} = k$ and is equal to zero otherwise. 
\end{notation} 

\begin{convention}
We adapt operations on numbers to vectors $\vec{u} \in \RR^n$ by applying the operation to each coordinate.  For example, $\floor{\vec{u}} = (\floor{u_1}, \dots, \floor{u_n})$ is the vector obtained by taking the floor of each component.  We write $\vec{u} \leq \vec{v}$ to mean $u_i \leq v_i$ for all $i$ and $\vec{u} \prec \vec{v}$ to mean $u_i < v_i$ for all $i$.
\end{convention}

\begin{prop} \label{t:keybackground} \label{monomial:Frobenius:powers}  \label{monomial:Frobenius:power:generators}
Let $I \subseteq S$ be a monomial ideal. Then with the notation above: 
\[ I^{[\frac{k}{q}]} = ( \vec{x}^{\floor{\frac{A\vec{u}}{q}}} \mid \vec{u} \in \NN^n, \norm{\vec{u}} = k, {\textstyle \binom{k}{\vec{u}}} \not\equiv 0 \!\!\! \pmod{p} ) \]
\end{prop}

\begin{proof}
By definition, $I^{[\frac{k}{q}]} = (I^{[k]})^{[\frac{1}{q}]}$ where 
\[ I^{[k]} = ( \vec{x}^{A\vec{u}} \mid \vec{u} \in \NN^n, \norm{\vec{u}} = k, {\textstyle \binom{k}{\vec{u}}} \not\equiv 0 \!\!\! \pmod{p} ) \]
by \cite[3.5]{Frobenius:powers}.  Since the monomials $\vec{x}^{\vec{b}}$ with $\vec{b} \prec q\ones$ form a basis for $S$ as a free $S^q$-module, we can compute the $q$-th root Frobenius power of an ideal $J = (f_1, \dots, f_n) \subseteq S$ as  $J^{[\frac{1}{q}]} = (f_{i, \vec{b}} \mid f = \sum_{\vec{b} \prec q\ones} f_{i,\vec{b}}^q\vec{x}^\vec{b})$ by \cite[2.5]{eth:roots}.  Applying this description to $I^{[k]}$ yields the claimed description for $I^{[\frac{k}{q}]}$.
\end{proof}

\begin{cor} \label{c:monomial:Frobenius:powers}
The Frobenius powers $I^{[\lambda]}$ of a monomial ideal $I \subseteq S$ are monomial ideals. 
\end{cor}

\begin{proof}
This is immediate from the above proposition since every Frobenius power $I^{[\lambda]}$ agrees with the Frobenius power of some $p$-adic rational exponent.
\end{proof}

\begin{cor}
Let $I \subseteq S$ be a monomial ideal as above.  If $x_j^2$ does not divide any $\vec{x}^{\vec{a}_i}$, then for every $0 \leq \frac{k}{q}<1$, $x_j$ does not divide any generator of $I^{[\frac{k}{q}]}$.
\end{cor}

\begin{proof}
By the above proposition, the generators of $I^{[\frac{k}{q}]}$ have the form $\vec{x}^{\floor{\frac{A\vec{u}}{q}}}$, for some vector $\vec{u} \in \NN^n$ such that $\norm{\vec{u}} = k$. 
Since each $\vec{x}^{\vec{a}_i}$ is not divisible by $x_j^2$, the $j$-th row of $A$ contains no entries greater than one, and we have $(A\vec{u})_j \leq \norm{\vec{u}} = k < q$.  Hence, the exponent of $x_j$ is $\floor{\frac{(A\vec{u})_j}{q}} = 0$.
\end{proof}

\begin{cor}\label{cor:sqfree}
If $I \subseteq S$ is a monomial ideal that contains a squarefree monomial, then $\lce(I) = 1$.
\end{cor}

\begin{proof}
If $\vec{x}^{\vec{a}_1}$ is a squarefree monomial generator of $I$ and $\vec{e}_1 \in \NN^n$ denotes the corresponding standard basis vector, then $I^{[\frac{k}{q}]}$ contains $\vec{x}^{\floor{\frac{A(k\vec{e}_1)}{q}}} = \vec{x}^{\floor{\frac{k\vec{a}_1}{q}}} = 1$.  Hence, $I^{[\frac{k}{q}]} = S$ for all $\frac{k}{q} < 1$.
\end{proof}

As a consequence of the above corollary, every squarefree monomial ideal has least critical exponent equal to one.  This is not surprising as the least critical exponent is supposed to be an analog of the F-pure threshold, and squarefree monomial ideals are F-pure.  However, a monomial ideal need not be F-pure in order to have least critical exponent equal to one.

\begin{example}
The ideal $I = (x^2, xy) \subseteq S = \kk[x, y]$ is not F-pure by Fedder's criterion \cite[1.12]{Fedder} since
\[ (I^{[p]}: I) = (x^{2p-2}, x^{p-2}y^p) \cap (x^{2p-1}, x^{p-1}y^{p-1}) = (x^{2p-1}, x^{2p-2}y^{p-1}, x^{p-1}y^p) \subseteq \m^{[p]} \]
However, $\lce(I) = 1$ because $I$ contains a square-free monomial.
\end{example}

\begin{defn}\label{d:lambdab}
For any monomial $\mathbf{x^{b}}\in S$, we define the \emph{critical exponent of $\mathbf{x^{b}}$} as
\[ \lambda_\vec{b}(I) = \sup\{ \lambda \in \RR_{\geq 0} \mid \vec{x}^\vec{b} \in I^{[\lambda]}\}. \]
Since $I^{[k]} \subseteq I^k$ is generated in degrees at least $k$, it is clear that $\vec{x}^\vec{b} \notin I^{[k]}$ for $k = \norm{\vec{b}} + 1$ so that the above supremum is always finite.  
\end{defn}

\begin{rmk}
We note that $\vec{x}^\vec{b} \notin I^{[\lambda_\vec{b}]}$ by the right constancy of Frobenius powers so that $\lambda_\vec{b}(I)$ is in fact a critical exponent of $I$.  By \cite[2.5]{Frobenius:powers:of:some:monomial:ideals}, the above definition coincides with what in that paper is called 
\[ \crit(I, \vec{b} + \ones) = \sup \{ \lambda \in \RR_{\geq 0} \mid I^{[\lambda]} \nsubseteq (x_1^{b_1+1}, \dots, x_m^{b_m + 1}) \} \]
In particular, we note that 
\[ \lce(I) = \lambda_\vec{0}(I) = \crit(I, \ones) = \sup\{ \lambda \in \RR_{\geq 0} \mid I^{[\lambda]} \nsubseteq (x_1, \dots, x_m) \} \]
and since $I^{[1]} = I \subseteq (x_1, \dots, x_m)$, it follows that $0 < \lce(I) \leq 1$.
\end{rmk}

\begin{prop}
Every critical exponent of a monomial ideal $I \subseteq S$ is of the form $\lambda_\vec{b}(I)$ for some $\vec{b} \in \NN^n$.
\end{prop}  

\begin{proof}
Although we do not assume that $I$ is $\m$-primary, the same argument as in \cite[2.6]{Frobenius:powers:of:some:monomial:ideals} shows that every critical exponent of $I$ is of the form $\crit(I, \vec{a})$, except that a priori we only have $\vec{a} \in \NN^n$ with $\vec{a} \neq \vec{0}$ instead of $\vec{a} \succ \vec{0}$.  However, since we know that the Frobenius power $I^{[\lambda]}$ is a monomial ideal by Corollary~\ref{c:monomial:Frobenius:powers}, it is easily seen that $I^{[\lambda]} \nsubseteq (x_1^{a_1}, \dots, x_n^{a_n})$ if and only if the monomial $m_\vec{a} = \prod_{a_i > 0} x^{a_i-1}$ is contained in $I^{[\lambda]}$.  Taking $\vec{\tilde{a}} = \max(\vec{a}, \ones)$, it is clear that $m_\vec{\tilde{a}} = m_\vec{a}$ so that $I^{[\lambda]} \nsubseteq (x_1^{a_1}, \dots, x_n^{a_n})$ if and only if $I^{[\lambda]} \nsubseteq (x_1^{\tilde{a}_1}, \dots, x_n^{\tilde{a}_n})$.  Hence, we have $\crit(I, \vec{a}) = \crit(I, \vec{\tilde{a}}) = \lambda_\vec{b}(I)$ for $\vec{b} = \vec{\tilde{a}} - \ones$.
\end{proof}

\begin{rmk} \label{rmk:skoda}
Due to their close relationship with test ideals of principal ideals, Skoda's Theorem for Frobenius powers \cite[3.17]{Frobenius:powers} implies that every critical exponent $\lambda > 0$ satisfies that $\lambda - \ceil{\lambda} + 1$ is also a critical exponent in the interval $(0, 1]$.  Consequently, we need only concern ourselves with finding critical exponents in this interval. We will recover this result using our techniques in Observation~\ref{obs:recovery}.
\end{rmk}

\subsection{Addition Base $p$ and the Sierpinski Simplex}

\paragraph{}  Our next task is to shed some light on the condition $\binom{k}{\vec{u}}\not\equiv 0 \pmod{p}$.  There is a useful interpretation in terms of the base $p$ expansion of $\vec{u}$.

\begin{lemma} \label{prop:basepaddition} 
Suppose $\norm{\vec{u}} = k$.  We have $\binom{k}{\vec{u}} \not\equiv 0 \pmod{p}$ if and only if the addition $\sum u_i = k$ involves no carries in base $p$.  That is, if we write $\vec{u} = (u_1, \dots , u_n)$ and write each $u_i = \sum_j v_{i,j}p^j$, then for all $j$ we have $\sum_i v_{i,j} < p$.
\end{lemma}

\begin{proof}
For an integer $m$, let $\nu_{p}(m)$ represent the number of times $m$ is divisible by $p$.  Then, if the base $p$ expansion of $m$ is $m=\sum c_{i}p^{i}$, we have $\nu_{p}(m!)=\sum_{i\neq 0}c_{i}(1+p+p^{2}+\dots + p^{i-1})$.

Now observe that $\nu_{p}\binom{k}{\mathbf{u}}=\nu_{p}(k!) - \sum \nu_{p}(u_{i}!)$.  Writing $k=\sum a_{j}p^{j}$ in base $p$, we have $a_{j}=\sum v_{i,j} + c_{j-1}-pc_{j}$, where $c_{j}$ is the number of carries in the $p^{j}$ place.  We compute $\nu_{p}\binom{k}{\mathbf{u}}=\sum c_{j}$.  Thus, $\binom{k}{\mathbf{u}}\equiv 0 \pmod{p}$ if and only if $\nu_{p}\binom{k}{\mathbf{u}}\geq 1$ if and only if there are carries in the addition.
\end{proof}

The description in terms of addition base $p$ allows us to translate from vectors of integers $\vec{u}$ with $\norm{\vec{u}} = k$ to vectors of $p$-adic rational numbers $\vec{u}$ with $\norm{\vec{u}} = \frac{k}{q}$.  Theorem \ref{t:keybackground} becomes the following:

\begin{prop}\label{t:basepaddition}
Let $I \subseteq S$ be a monomial ideal, $\vec{x}^\vec{b}$ be a monomial in $S$, and $\frac{k}{q} \in \ZZ[\frac{1}{p}]$.  Then $\vec{x}^\vec{b} \in I^{[\frac{k}{q}]}$ if and only if the set
\[ Q_{\frac{k}{q},\vec{b}}(I) = \{ \vec{u} \in \ZZ[\tfrac{1}{p}]^n : \norm{\vec{u}} = \tfrac{k}{q}, {\textstyle \sum u_i} \; \text{adds without carries}, \floor{A\vec{u}} \leq \vec{b}\}. \]
is nonempty.
\end{prop}

\begin{definition}\label{d:admissible}
  We say that a vector $\vec{u}\in \mathbb{R}^{n}$ is \emph{admissible} if it is possible to choose base $p$ representations of every entry $u_{i}$ in such a way that $\sum u_{i}$ adds without carries.
\end{definition}

\begin{remark}\label{r:admissible}
  Because $p$-adic fractions have two base $p$ representations, the definition of admissibility is subtler than it looks.  For example, if $p=2$, the vector $(\frac{1}{2}, \frac{1}{2})$ is admissible:  While the base 2 addition (.1 + .1 = 1) involves a carry, the addition $(.1 + .0\overline{1}=.\overline{1})$ does not.   
\end{remark}

The notion of admissibility allows us to begin describing the critical exponents of $I$ geometrically.  By Theorem \ref{t:basepaddition}, we have

\begin{prop} \label{p:lambdabmax}
For any exponent vector $\vec{b} \in \NN^n$, we have
\[ \lambda_\vec{b}(I)=\sup \{ \norm{\vec{u}} : \vec{u} \in \RR^n, \vec{u} \text{ is admissible}, \floor{A\vec{u}} \leq \vec{b} \} \]
\end{prop}

\begin{proof}
Let $\ell$ be the supremum on the right-hand side.  Then for any $\epsilon$, there exists $\vec{u}$ with $\|\vec{u}\|>\ell - \frac{\epsilon}{2}$.  Choose $e$ such that $\frac{1}{p^{e}}<\frac{\epsilon}{2}$, and let $\vec{u}_{e}$ be the vector obtained by writing $\vec{u}$ in base $p$ so that it adds with no carries, and then truncating every entry of $\vec{u}$ after $e$ decimal places.  Since $\vec{u}$ adds without carries, we have $\|\vec{u}\|-\|\vec{u}_{e}\|<\frac{1}{p^{e}}$.  Thus $\ell<\|\vec{u}_{e}\|+\epsilon<\lambda_{\vec{b}}+\epsilon$.  So $\ell \leq \lambda_{\vec{b}}$. Meanwhile, $\lambda_{\vec{b}}\leq \ell$ is immediate, since it is the supremum of a smaller set.
\end{proof}

\begin{example} \label{e:openfractalonly}
It is tempting to try to replace the supremum in Proposition \ref{p:lambdabmax} with a maximum by invoking compactness.  However, it is incorrect to claim that $\lambda_\vec{b}$ is equal to the maximum value of $\norm{\vec{u}}$, over all admissible $\vec{u}$ satisfying $\floor{A\vec{u}} \leq \vec{b}$.  Consider $I = (x^2y^2, y^3z^3) \subseteq S = \kk[x, y, z]$ as in Example \ref{e:babypolytopes}, and set $\vec{b} = (0,0,0)$.  We have $\lambda_\vec{b} = \frac{1}{2}$, realized by a sequence of vectors $\vec{u}$ approaching $(\frac{1}{2},0)$ from the left.  (The terms of this sequence depend on the characteristic.  If $p=2$, they are truncations of the binary vector $(.0\bar{1}, 0)_2$.)  But this sequence does not contain its limit point $\vec{v}=(\frac{1}{2},0)$, and in fact $\lfloor A\vec{v}\rfloor=(0,1,0)$ is not less than $\vec{b}$.
\end{example}

We now turn briefly to understanding the set of admissible vectors.

\begin{definition}\label{d:fractal}
Fix a prime $p$ and a dimension $d$.  The \emph{(open) $(p,d)$-Sierpinski simplex} is the set $\SS_{p,d}$ consisting of
  all $d$-tuples $(x_{1},\dots, x_{d})$ such that $0\leq x_{i}\leq 1$, each
  $x_{i}$ is a terminating decimal in base $p$, and these decimals add
  without a carry.

  The \emph{(closed) $(p,d)$-Sierpinski simplex} is the set
  $\overline{\SS}_{p,d}$ consisting of all real admissible $d$-tuples $(x_{1},\dots, 
  x_{d})$ such that $0\leq x_{i}\leq 1$.
\end{definition}

\begin{remark}
The distinction between the open and closed Sierpinski simplices is
not simply the distinction between $p$-adic fractions and real
numbers.  The non-uniqueness of decimal representations for
terminating decimals allows the closed Sierpinski simplex to contain
many rational points that are missing from the open simplex.  For
example, $\SS_{2,2}$ does not contain
$(\frac{1}{2},\frac{1}{2})=(.1,.1)$ because the sum $.1+.1=1$ involves
a carry.  However, $\overline{\SS}_{2,2}$ does contain this point,
because we may choose to write it as $(.1, .0\overline{1})$, and the
sum $.1+.0\overline{1}=.\overline{1}\dots$ does not require a carry.
\end{remark}

\begin{remark}
The sets $\SS_{p,d}$ and $\overline{\SS}_{p,d}$ are fractals.
$\overline{\SS}_{2,2}$ is the familiar Sierpinski gasket, and
$\overline{\SS}_{p,d}$ has dimension $\log_{p}\binom{p+d-1}{d}$.  We
provide several iterative methods for building these fractals.

\begin{description}
\item[Method one:]  Set $X=\{(x_{1},\dots, x_{d})\in \mathbb{Z}_{\geq
  0}^{d}:\sum x_{i}<p\}$.  Put $S_{1}=\{\frac{1}{p}x:x\in X\}$,
  $S_{2}=\{v+\frac{1}{p^{2}}w:v\in S_{1}, w\in X\}$, and in general
  $S_{i}=\{v+\frac{1}{p^{i}}w: v\in S_{i-1},w\in X\}$.  Then
  $\SS_{p,d}=\bigcup S_{i}$.

  Equivalently, $\SS_{p,d}=X+\frac{1}{p}\SS_{p,d}$.

\item[Method two:] Let $T_{0}$ be the unit hypercube $\{(x_{1},\dots,
  x_{d}):0\leq x_{i}\leq 1\}$.  Subdivide $T_{0}$ into $p^{d}$
  congruent hypercubes of side length $\frac{1}{p}$ in the standard way.
  Then delete all the smaller cubes that lie entirely in the
  half-space $\sum x_{i}\geq 1$.  Call the result $T_{1}$.  Then
  replace each of the cubes in $T_{1}$ with a $\frac{1}{p}$-scale
  copy of $T_{1}$; the result is $T_{2}$.  In general, obtain $T_{i}$
  by replacing each of the surviving cubes in $T_{1}$ with a
  $\frac{1}{p^{i-1}}$-scale copy of $T_{i-1}$ (or, equivalently,
  replace each of the surviving cubes in $T_{i-1}$ with a copy of
  $T_{1}$).  Then $\overline{\SS}_{p,d}=\cap T_{i}$.

\item[Method three:]  Let $W_{0}$ be the unit $d$-simplex (the convex
  hull of the origin and the $d$ standard basis vectors).  Divide each
  of the edges into $p$ equal segments, and draw in all hyperplanes
  parallel to the sides of $W_{0}$ and through the new vertices.  This
  divides $W_{0}$ into $\binom{p+d-1}{d} + \binom{p+d-2}{d}$ congruent
  sub-simplices, of which $\binom{p+d-1}{d}$ are oriented correctly.
  Delete the backwards simplices and call the result $W_{1}$.  In
  general, obtain $W_{i+1}$ by replacing each simplex in $W_{i+1}$
  with a scaled-down copy of $W_{1}$.  
\end{description}
\end{remark}

The closed Sierpinski simplex consists of all admissible vectors, but the open simplex is necessary to compute the jumping numbers.

Proposition \ref{p:lambdabmax} becomes the following:

\begin{thm}\label{t:lambdabwithfractal}
Let $I$ be a monomial ideal and $\vec{b} \in \NN^m$.  Then 
\[ \lambda_\vec{b}(I) = \sup\{ \norm{\vec{u}} : \vec{u} \in \SS_{p,n}, \floor{A\vec{u}} \leq \vec{b}\}. \]
\end{thm}

\subsection{Truncations and Witnesses}

\paragraph{}  In general, the real number $\lambda_\vec{b}$ is a critical exponent for $I$ if and only if, for every $p$-adic rational number $\ell=\frac{k}{q} < \lambda_\vec{b}$, there exists a vector $\vec{u}_\ell \in \frac{1}{q}\NN^n\cap \SS_{p, n}$ such that $\norm{\vec{u}} = \ell$ and $\floor{A\vec{u}_{\ell}} \leq \vec{b}$.  These $\vec{u}_\ell$ will have a limit point $\vec{u} \in \bar{\SS}_{p, n}$.
We are interested in finding this limit point $\vec{u}$ to avoid dealing with infinite sequences, but there are two pitfalls.

\begin{itemize}
\item If $\floor{A\vec{u}_\ell} \leq \vec{b}$, then the entries of $A\vec{u}$ are limits of the corresponding entries of $A\vec{u}_\ell$, which may converge to the next larger integer.  In particular, we may not have $\floor{A\vec{u}} \leq \vec{b}$.  The best we can do is to drop the floors, resulting in $A\vec{u} \leq \vec{b} + \ones$.

\item There may exist $\vec{v} \in \bar{\SS}_{p, n}$ satisfying $A\vec{v} \leq \vec{b} + \ones$ which are not limit points of a suitable sequence of $\vec{v}_{e}$.  For example, if $I = (x^2y^2,y^3z^{3}) \subseteq S = \kk[x, y, z]$, $\vec{b} = (0, 1, 0)$, and $p = 3$, then $\vec{v} = (\frac{1}{2}, \frac{1}{3})=(.\bar{1}, .1)_{3}\in \SS_{3,2}$ satisfies $A\vec{v} \leq (1,2,1)$ but is not a limit of suitable $\vec{v}_\ell$.  (In fact, in this case, $\lambda_\vec{b} = \frac{2}{3} \neq \norm{\vec{v}}$.)  We want to ignore such $\vec{v}$. 
\end{itemize}

We introduce new notation in hopes of addressing these issues.

\begin{definition}\label{def:stuartscott}
A vector $\vec{v} \in \SS_{p, n}$ is a \emph{witness} for the critical exponent $\lambda = \lambda_\vec{b}(I)$ if $A\vec{v} \prec \vec{b} + \ones$ and $\norm{\vec{v}} = \lambda$.  More generally, a vector $\vec{v} \in \SS_{p, n}$ is an \emph{$e$-witness} for $\lambda$ if $\vec{v} \in \frac{1}{p^e}\NN^n$, $A\vec{v} \prec \vec{b} + \ones$, and $\lambda - \frac{1}{p^e} \leq \norm{\vec{v}} < \lambda$.  We say that a sequence of vectors $\{\vec{v}_e\}$ is a \emph{convergent family of witnesses} if for each $e$, $\vec{v}_e$ is an $e$-witness, and every component of $\vec{v}_{e+1}$ agrees with the corresponding component of $\vec{v}_e$ to $e$ decimal places.
\end{definition}

We find it easier to think of $e$-witnesses in terms of truncation.

\begin{definition}\label{def:truncation_witness}
For any positive real number $z$ with nonterminating base $p$ decimal expansion
$z=z_0.z_1z_2\dots$, we define $\trunc_{e}(z)$ to be the truncation of this decimal after $e$ places, that is, $\trunc_{e}(z) = z_0.z_1z_2 \dots z_e.$
\end{definition}

\begin{remark}\label{rem:truncation_witness}
With this definition, $\vec{v}$ is an $e$-witness for $\lambda_{\vec{b}}$ if $\vec{v} \in \SS_{p,n}$, $A\vec{v} \prec \vec{b}+\ones$, and $\norm{\vec{v}} = \trunc_{e}(\lambda)$.
\end{remark}

\begin{remark}\label{rem:truncation_strictly_less}
For all $z$, the truncation of $z$ is strictly less than $z$ because we choose the nonterminating representation of $z$. For example, if $p=2$, and $z=\frac{1}{2}$, we write $z=0.0\overline{1}$, so $\trunc_2\left(\frac{1}{2}\right)=0.01=\frac{1}{4}$.
\end{remark}

\begin{lemma}\label{lem:ewitnessesexist}
For all $e$, $e$-witnesses exist.
\end{lemma}

\begin{proof}
Choose a $\vec{v} \in \SS_{p,n}$ with $A\vec{v} \prec \vec{b}+\ones$ and $\lambda-\norm{\vec{v}} < \lambda -\trunc_{e}(\lambda)$. By Theorem~\ref{t:lambdabwithfractal}, this is possible because $\lambda$ is the supremum over all such $\vec{v}$. Observe that $\lambda$ agrees with $\norm{\vec{v}}$ to $e$ decimal places. Now let $\vec{v}_e$ be obtained from $\vec{v}$ by truncating each entry after $e$ decimal places. Because the entries of $\vec{v}$ add without carries, the rational numbers $\norm{\vec{v}}$ and $\norm{\vec{v}_e}$ agree to $e$ decimal places.
\end{proof}


\begin{proposition}\label{prop:conv_family}
Convergent families of witnesses exist. If $\{\vec{v}_e\}$ is a convergent family of witnesses for $\lambda_{\vec{b}}$, set $\vec{v} = \lim_{e \to \infty} \vec{v}_e$. Then $\norm{\vec{v}} = \lambda_{\vec{b}}$, and $\norm{\vec{v}_e} = \trunc_{e}(\lambda_{\vec{b}}) = \trunc_e(\norm{\vec{v}})$. 
\end{proposition}

\begin{proof}
For existence, observe that the set of nonnegative vectors $\vec{v}$ satisfying $A\vec{v} \prec \vec{b} + \ones$ is bounded, so Cauchy sequences exist. For fixed $e$, the set of $e$-witnesses is discrete, so any Cauchy sequence contains an $e$-witness for arbitrarily large $e$. Since all vectors $\vec{v}$ are in $\RR^n$, the Cauchy sequences converge.
\end{proof}

\begin{rmk}\label{rmk:family}
If $\vec{v}$ is a vector of decimals written in base $p$, not all terminating, define $\tau_{e}(\vec{v})$ to be the vector obtained by truncating each entry after $e$ decimal places. (Note that $\tau_{e}(\vec{v}) \neq \trunc_{e}(\vec{v})$ because some of the entries may be terminating decimals.) By abuse of notation, if $\norm{\vec{v}} = \lambda_\vec{b}$ without carries, then $\{\tau_{e}(\vec{v})\}$ is a family of witnesses. We also refer to $\vec{v}$ as a family of witnesses.
\end{rmk}

\begin{cor}
With the notation as above, we have
\[ \lambda_\vec{b}(I) = \sup\{ \norm{\vec{z}} \mid \vec{z} \in \SS_{p, n}, A\vec{z} \prec \vec{b} + \ones \} 
= \sup\{ \norm{\vec{z}} \mid \vec{z} \in \bar{\SS_{p, n}}, A\vec{z} \prec \vec{b} + \ones \} \]
\end{cor}

\begin{proof}
The set on the left contains every $e$-witness, and the set on the right contains the limit of every convergent sequence of $e$-witnesses.
\end{proof}

\begin{example}\label{e:babypolytopes}
  Let $I=(x^{2}y^{2}, y^{3}z^{3}) \subseteq S = \kk[x, y, z]$.  Figure \ref{f:babypolytopes} divides the rectangle $[0,q]\times [0,q]$ into several regions.  Suppose we are interested in computing the least critical exponent of $I$, i.e., determining when $1=\mathbf{x}^{(0,0,0)}\in I^{[\frac{k}{q}]}$.  The matrix $A$ is $\begin{bmatrix}2&0\\2&3\\0&3\end{bmatrix}$.  A vector $\mathbf{u}$ satisfies $A\mathbf{u}\leq p^{e}(\mathbf{b + 1}) - \mathbf{1}$ if it is to the left of the dashed line labeled $a^{1}$, below the dashed line labeled $c^{1}$, and below and to the left of the dashed line labeled $b^{1}$.

    The triangle labeled ``1'' in the lower left corner is thus the collection of all vectors $\mathbf{u}$ satisfying $A\mathbf{u}\leq p^{e}(\mathbf{b + 1}) - \mathbf{1}$.  Thus $1\in I^{[\frac{k}{q}]}$ if and only if there is some $\mathbf{u}$ in this triangle with $\|\mathbf{u}\|=k$ and $\binom{k}{\mathbf{u}}\not\equiv 0$ (mod $p$).  In fact, (independent of characteristic), the vector $\mathbf{u}=(k,0)$ accomplishes this whenever $k<\frac{q}{2}$.  We conclude that the least critical exponent is thus the supremum, over all $q$, of the fractions $\frac{k}{q}$ with $k<\frac{q}{2}$; that is, the least critical exponent of $I$ is equal to $\frac{1}{2}$.

    We will see later that the condition ``$\binom{k}{\mathbf{u}}\not\equiv 0$ (mod $p$)'' is considerably more interesting, even for this ideal.  The other jumping numbers will turn out to depend on the characteristic.
\end{example}

  \begin{figure}[hbtp]\label{f:babypolytopes}
    \begin{center}
\scalebox{0.5}{
\begin{tikzpicture}
\draw[gray, very thin] (0, 0) grid (12.9, 12.9);
\draw[->] (-1,0) -- (13,0);
\draw[->] (0, -1) -- (0, 13); 
\node[above] (u2) at (0, 13) {$u_2$};
\node[right] (u1) at (13, 0) {$u_1$};
\node[below] at (6, 0) {$\frac{q}{2}$};
\node[below] at (12, 0) {$q$};
\draw[thick, dashed] (6, 0) -- (6, 13);
\node[above] at (6, 13) {$a^1$}; 
\draw[thick, dashed] (12, 0) -- (12, 13);
\node[above] at (12, 13) {$a^2$}; 
\node[left] at (0, 4) {$\frac{q}{3}$};
\node[left] at (0, 8) {$\frac{2q}{3}$};
\node[left] at (0, 12) {$q$};
\draw[thick, dashed] (0, 4) -- (13, 4);
\node[right] at (13, 4) {$c^1$}; 
\draw[thick, dashed] (0, 8) -- (13, 8);
\node[right] at (13, 8) {$c^2$}; 
\draw[thick, dashed] (0, 12) -- (13, 12);
\node[right] at (13, 12) {$c^3$};
\draw[thick, dashed] (0, 4) -- (6, 0);
\node[above right] at (3, 2) {$b^1$}; 
\draw[thick, dashed] (0, 8) -- (12, 0);
\node[above right] at (6, 4) {$b^2$}; 
\draw[thick, dashed] (0, 12) -- (13, 10/3);
\node[above right] at (6, 8) {$b^3$}; 
\node at (1.5, 1.5) {\Large $1$};
\node at (4.5, 2.5) {\Large $b$};
\node at (8, 1.5) {\Large $ab$};
\node at (10.5, 2.5) {\Large $ab^2$};
\node at (1.5, 5.5) {\Large $bc$};
\node at (1.5, 9.5) {\Large $b^2c^2$};
\node at (8, 5.5) {\Large $ab^2c$};
\node at (4.5, 6.5) {\Large $b^2c$};
\node at (10.5, 6.5) {\Large $ab^3c$};
\node at (13, 1.5) {\Large $a^2b^2$};
\end{tikzpicture}
}
\end{center}
\caption{Subdivision of a $q\times q$ rectangle into polytopes for $I=(a^{2}b^{2},b^{3}c^{3})$}
  \end{figure}
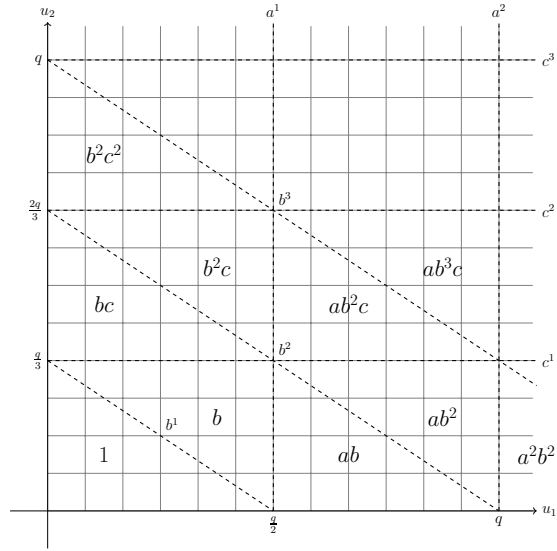

\begin{observation} \label{obs:recovery}
We recover the content of Remark~\ref{rmk:skoda} using our methods, specifically that if $\lambda > 1$ is a critical exponent, then $\lambda-1$ is as well. Choose a monomial $x^{\vec{b}}$ with $\lambda=\lambda_{\vec{b}}$. Because $\lambda>1$, $x^{\vec{b}} \in I$, and there exists $x^{\vec{a}}$ among the minimal generators of $I$ dividing $x^{\vec{b}}$. Then $\lambda-1 = \lambda_{\vec{b}-\vec{a}}$.
\end{observation}

\section{An Algorithm for Computing Critical Exponents}\label{section:algo}

\paragraph{} Our strategy for computing the critical exponent $\lambda_{\vec{b}}(I)$ associated to a vector $\vec{b} \in \NN^m$ is to recursively compute its base $p$ decimal expansion to $e$ decimal places of accuracy. 
We begin by defining an (infinite) process that computes all $e$-witnesses to $\lambda_{\vec{b}}$ for all $e$. Later we show in Algorithm~\ref{algo:second} that this can be adapted to a terminating algorithm. In the Appendix, we demonstrate Algorithm~\ref{algo:second}, computing some critical exponents.

\begin{algo}\label{algo:first}
Fix a monomial $\vec{x}^\vec{b} \notin I$ and an integer $e \geq 0$. Starting with $\mathcal{L}_0 = \{\vec{0}\}$, inductively compute the set $\mathcal{L}_e$ of all $e$-witnesses to $\lambda_\vec{b}$ as follows:
\begin{enumerate}
\item For each $e \geq 1$ and $\vec{u} \in \mathcal{L}_{e-1}$, compute the remainder vector $\vec{r} = \vec{b} + \ones - A\vec{u}$.
\item Find all solutions $\vec{v} \in \NN^n$ maximizing $\norm{\vec{v}}$ subject to the constraints that $A \vec{v} \prec p^{e}\vec{r}$ and $\norm{\vec{v}} < p$.
\item Append each $\vec{u}+\frac{1}{p^e}\vec{v}$ to $\mathcal{L}_{e}$.
\item After doing this for all $\vec{u}$, compute $\lambda_{e} = \max\{\norm{\vec{w}} : \vec{w} \in \mathcal{L}_{e}\}$, and remove from $\mathcal{L}_{e}$ every $\vec{w}$ with $\norm{\vec{w}} < \lambda_{e}$.
\end{enumerate}
\end{algo}

\begin{rmk}
We make some simple observations about the preceding algorithm:
\begin{itemize}
\item The computation of all $\vec{v}$ in Step 2 terminates because there are finitely many integer vectors $\vec{v}$ with $\norm{\vec{v}} < p$.
\item Since the base $p$ expansion of every $\vec{u} \in \mathcal{L}_{e-1}$ terminates after $e - 1$ decimal places by induction, Step 3 effectively appends the entries of each vector $\vec{v}$ to the $e$-th decimal places of $\vec{u}$.
\item Every element $\vec{u} + \frac{1}{p^e}\vec{v}$ of $\mathcal{L}_{e}$ is contained in $\SS_{p,n}$ since $\vec{u} \in \SS_{p, n}$ by induction and $\norm{\vec{v}} < p$ so that computing the norm of $\vec{u} + \frac{1}{p^e}\vec{v}$ doesn't involve a carry.
\end{itemize}
\end{rmk}

If we had infinite time and space, we could use Algorithm~\ref{algo:first} to compute $\lambda_\vec{b}$. It is the limit of $\lambda_e$ as $e$ goes to infinity.  To compute $\lambda_\vec{b}$ with finite resources, we need to modify the algorithm to detect when it starts looping. For inspiration, we turn to a familiar algorithm with this capability, long division.

\begin{example}\label{longdiv}
We use long division to find the decimal expansion of $\frac{1}{22}$. Long division first divides 22 into 1, computing a quotient of 0 and a remainder of 1. It then multiplies the remainder by 10 and divides by 22 again, producing a quotient of 0 and a remainder of 10. At the next step, we get a quotient of 4 and a remainder of 12, followed by a quotient of 5 and a remainder of 10. At this point, the algorithm recognizes that the new remainder has appeared before, so all steps from the previous remainder of 10 will repeat. Thus all future quotients will repeat in a pattern of 4, 5, 4, 5, \dots. We conclude that $\frac{1}{22} =0.0\overline{45}$. 
\end{example}

In order to modify Algorithm~\ref{algo:first}, we need to keep track of remainders at every step. Unfortunately, there are some complications because there may be multiple $\vec{v}$ for a given $\vec{u}$, so we are not forced to repeat in the same pattern. Hence we need a lemma to show that these complications do not matter in computing $\lambda_{\vec{b}}$.

\begin{lemma}\label{l:replacelarge}
With the notation of Algorithm~\ref{algo:first}, since $A$ is fixed, and every $\vec{v}$ satisfies $\norm{\vec{v}} < p$, there exists an integer $\Omega$ such that every entry of $A\vec{v}$ is at most $\Omega$. Consequently, if any entry of $\vec{r}$ is greater than $\Omega$, we may replace that entry with $\Omega$ without changing the result of the algorithm. In particular, after this modification, there are only finitely many possible remainder vectors.
\end{lemma}

\begin{definition}\label{def:ru}
Given an $e$-witness $\vec{u}$, set $R_{\vec{u}}=(\vec{r}_0,\vec{r}_1, \dots, \vec{r}_{e-1})$ to be the list of remainder vectors that arose in the computation of $\vec{u}$ in Algorithm~\ref{algo:first}. (Note that $\vec{r}_0 = \vec{b}+\ones$.)
\end{definition}

For the sake of completeness, in the following remark, we explicitly describe the $\vec{r}_i$.

\begin{remark}\label{r:remainder}
Because Algorithm~\ref{algo:first} computes each $\vec{u}$ exactly once, the remainder vector depends only on $\vec{u}$. Specifically, if we write
\[
\vec{u} = \begin{pmatrix} u_1 \\ u_2 \\ \vdots \\ u_n \end{pmatrix}
= \begin{pmatrix} 0.u_{1,1} u_{1,2} \dots u_{1,e} \\ 0.u_{2,1} u_{2,2} \dots u_{2,e} \\ \vdots \\ 0.u_{n,1} u_{n,2} \dots u_{n,e} \end{pmatrix} \text{in base } p,\] we set 
\[
\trunc_i(u) = 
\begin{pmatrix}
0.u_{1,1} u_{1,2} \dots u_{1,i} \\ 0.u_{2,1} u_{2,2} \dots u_{2,i} \\ \vdots \\ 0.u_{n,1} u_{n,2} \dots u_{n,i}
\end{pmatrix}.
\]
Then
\[
\vec{r}_i = p^i (\vec{b}+\ones-A\trunc_i(\vec{u})) = p\vec{r}_{i-1}-A\begin{pmatrix} u_{1,i} \\ \vdots \\ u_{n,i} \end{pmatrix}.
\]
\end{remark}

\begin{algo}\label{algo:second}
Fix a monomial $\vec{x}^\vec{b} \notin I$. We will compute $\lambda_\vec{b}$ and a convergent family of witnesses $\{\vec{v}_e\}$. We initialize $\Lambda$, the set of repeating candidates, to be $\Lambda =\varnothing$. At each base $p$ decimal place, we compute the set $\mathcal{L}_e=\{(\vec{u},R_{\vec{u}})\}$ of not yet repeating witnesses and their corresponding remainder vectors as follows.

Inductively, compute $\mathcal{L}_{e-1}$. For each pair $(\vec{u},R_{\vec{u}}) = (\vec{u},(\vec{r}_0,\vec{r}_1,\dots,\vec{r}_{e-2})) \in \mathcal{L}_{e-1}$, do the following.

\begin{enumerate}
\item Compute the next remainder vector $\vec{r}_{e-1}=\vec{b}+\ones-A\vec{u}$.
\item If $\vec{r}_{e-1} = \vec{r}_c$ for some $c<e-1$, then add
\[\widetilde{\vec{u}}=
\begin{pmatrix}
0.u_{1,1} u_{1,2} \dots u_{1,c-1} \overline{u_{1,c} \dots u_{1,e-1}} \\ 0.u_{2,1} u_{2,2} \dots u_{2,c-1} \overline{u_{2,c} \dots u_{2,e-1}} \\ \vdots \\ 0.u_{n,1} u_{n,2} \dots u_{n,c-1} \overline{u_{n,c} \dots u_{n,e-1}}
\end{pmatrix}
\] to $\Lambda$.
\item If not, find all solutions $\vec{v} \in \mathbb{N}^n$ to $A \vec{v} \prec p^{e}\vec{r}$ such that $\norm{\vec{v}}<p$ maximizing $\norm{\vec{v}}$ (subject to these constraints), and append each pair $(\vec{u}+\frac{1}{p^e} \vec{v},(\vec{r}_0, \dots, \vec{r}_{e-1}))$ to the list of candidates $\mathcal{L}_{e}$.
\end{enumerate}

After doing this for all $\vec{u} \in \mathcal{L}_{e-1}$, compute $\lambda_{e} = \max\{\norm{\vec{w}} : \vec{w} \in \mathcal{L}_{e}\}$, and delete from $\mathcal{L}_{e}$ every $(\vec{w},R_{\vec{w}})$ with $\norm{\vec{w}} < \lambda_{e}$.

Since there are finitely many possible remainder vectors by Lemma~\ref{l:replacelarge}, for sufficiently large $e$, $\mathcal{L}_e$ will be empty. At each stage, we add only finitely many vectors to $\Lambda$, so $\Lambda$ is finite at the end of this process. Choose $\vec{u}_{\ast} \in \Lambda$ maximizing $\norm{\vec{u}_{\ast}}$, and output $\lambda_{\vec{b}} = \norm{\vec{u}_{\ast}}$ and the family of witnesses $\vec{u}_{\ast}$.
\end{algo}

We conclude the proof of the correctness of Algorithm~\ref{algo:second} with the following lemma.

\begin{lemma}\label{l:getmax}
Suppose $(\vec{u}_{\ast},\norm{\vec{u}_{\ast}})$ is the output of Algorithm~\ref{algo:second}, and $\vec{v}$ is an $e$-witness to $\lambda_{\vec{b}}$. Then $\norm{\vec{u_{\ast}}} > \norm{\vec{v}}$.
\end{lemma}

\begin{proof}
Compute $R_{\vec{v}}$, which is possible by Remark~\ref{r:remainder}. If $R_{\vec{v}}$ has all unique entries, then Algorithm~\ref{algo:second} enters $\vec{v}$ into $\mathcal{L}_e$. Every $(\vec{u},R_{\vec{u}})$ in $\mathcal{L}_e$ has $\norm{\vec{u}} \ge \norm{\vec{v}}$, so all entries $\vec{w} \in \Lambda$ arising after the $(e+1)^\text{st}$ step have $\norm{\vec{w}} > \norm{\vec{v}}$.

If $R_{\vec{v}}$ does not have all unique entries, then $R_d=R_c$ for some $c < d \le e-1$. Then $\vec{w}=v_1v_2 \dots \overline{v_c v_{c+1} \dots v_{d-1}} \in \Lambda$, and $\norm{\vec{w}} > \norm{\vec{v}}$.
\end{proof}

We end the section with two corollaries. 

\begin{corollary}\label{cor:rational}
Let $I$ be a monomial ideal. Then all critical exponents of $I$ are rational.
\end{corollary}

\begin{cor} \label{cor:multbyp}
Let $I$ be a monomial ideal. Suppose that $\lambda = \lambda_{\vec{b}}$ is a critical exponent for $I$. Then (the fractional part of) $p \lambda$ is also a critical exponent for $I$.
\end{cor}

\begin{proof}
Suppose $\{\vec{v}_e\}$ is a convergent family of witnesses for $\lambda$. Then $\{p(\vec{v}_{e}-\vec{v}_{1})\}$ is a family of witnesses for $\lambda_{\vec{r}}$, where $\vec{r}=p(\vec{b}+\ones)-A\vec{v}_{1}-\ones$.
\end{proof}

\section{Examples of Frobenius Powers}\label{section:examples}

\begin{prop}
Let $S = \kk[x, y, z]$ and $I = (x^2y^2, y^dz^d)$ where $d \geq 2$. Then $\lce(I) = \frac{1}{2}$.
\end{prop}

\begin{proof}
First, we will show that $I^{[\frac{k}{q}]} = S$ for all $\frac{k}{q} < \frac{1}{2}$.  By Proposition \ref{monomial:Frobenius:power:generators}, we have 
\[ I^{[\frac{k}{q}]} = (x^{\floor{\frac{2u_1}{q}}}y^{\floor{\frac{2u_1 + du_2}{q}}}z^{\floor{\frac{du_2}{q}}} \mid u_1, u_2 \geq 0, u_1 + u_2 = k, {\textstyle \binom{k}{u_1}} \not\equiv 0 \!\!\! \pmod{p} ) \]
If $\frac{k}{q} < \frac{1}{2}$, we can take $u_1 = k$ and $u_2 = 0$ to see that $I^{[\frac{k}{q}]}$ contains  $x^{\floor{\frac{2k}{q}}}y^{\floor{\frac{2k}{q}}} = 1$.  

Next, we will compute $I^{[\frac{1}{2}]}$.  If $\ch(\kk) = 2$, it is immediate that $I^{[\frac{1}{2}]} = (xy, y^{\floor{\frac{d}{2}}}z^{\floor{\frac{d}{2}}})$, so we may assume that $\ch(\kk) \neq 2$.  In that case, we will show that $I^{[\frac{\ell}{q}]} = (y)$ for $\ell = \floor{\frac{q}{2}} + 1$ and all $q > d$.  Since $\ch(\kk) \neq 2$, we know that $q = 2\floor{\frac{q}{2}} + 1$ so that $\frac{\ell}{q} = \frac{2\floor{\frac{q}{2}} + 2}{2q} = \frac{q+1}{2q} = \frac{1}{2} + \frac{1}{2q} \to \frac{1}{2}$ as $q \to \infty$.  Hence, it will follow that $I^{[\frac{1}{2}]} = (b)$ so that $\lce(I) = \frac{1}{2}$.  Taking $u_1 = \ell - 1$ and $u_2 = 1$, we note that $\ell = \binom{\ell}{u_1} \not\equiv 0 \pmod{p}$, otherwise we would have $0 \equiv 2\ell = q +1 \equiv 1 \pmod{p}$.  And so, $I^{[\frac{\ell}{q}]}$ contains $x^{\floor{\frac{2(\ell -1)}{q}}}y^{\floor{\frac{2\ell+d-2}{q}}}z^{\floor{\frac{d}{q}}} = x^{\floor{\frac{2(\ell -1)}{q}}}y^{1 + \floor{\frac{d-1}{q}}}z^{\floor{\frac{d}{q}}} = y$ if $q > d$.  On the other hand, for any other $u_1, u_2 \geq 0$ with $u_1 + u_2 = \ell$ and $\binom{\ell}{u_1} \not\equiv 0 \pmod{p}$, we see that $x^{\floor{\frac{2u_1}{q}}}y^{\floor{\frac{2u_1 + du_2}{q}}}z^{\floor{\frac{du_2}{q}}} = x^{\floor{\frac{2u_1}{q}}}y^{\floor{\frac{2\ell + u_2}{q}}}z^{\floor{\frac{du_2}{q}}} = x^{\floor{\frac{2u_1}{q}}}y^{1 + \floor{\frac{1 + u_2}{q}}}z^{\floor{\frac{du_2}{q}}}$ is divisible by $b$ so that $I^{[\frac{\ell}{q}]} = (y)$ for all $q > d$.
\end{proof}

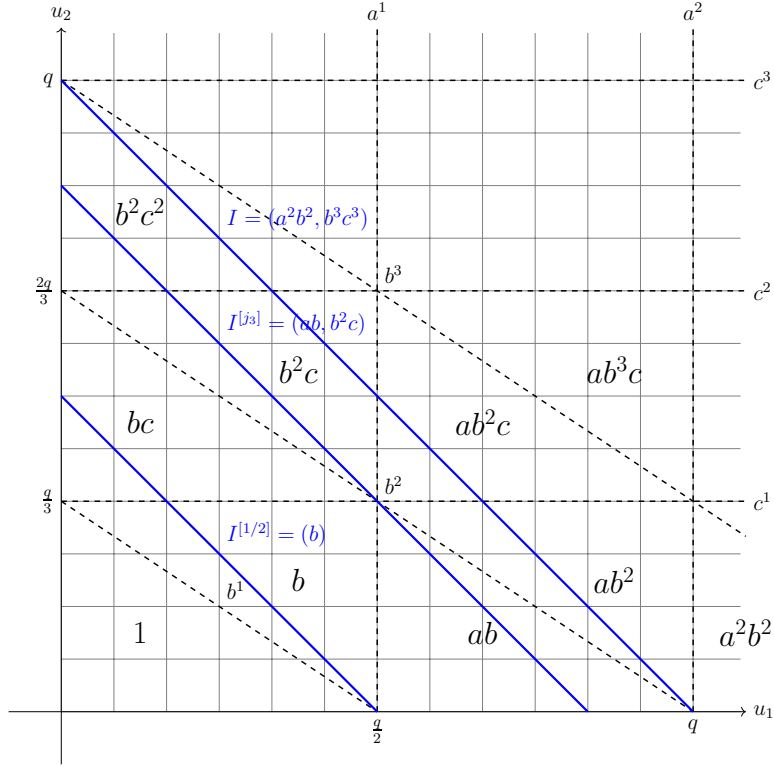
\begin{center}
\scalebox{0.7}{
\begin{tikzpicture}
\draw[gray, very thin] (0, 0) grid (12.9, 12.9);
\draw[->] (-1,0) -- (13,0);
\draw[->] (0, -1) -- (0, 13); 
\node[above] (u2) at (0, 13) {$u_2$};
\node[right] (u1) at (13, 0) {$u_1$};
\node[below] at (6, 0) {$\frac{q}{2}$};
\node[below] at (12, 0) {$q$};
\draw[thick, dashed] (6, 0) -- (6, 13);
\node[above] at (6, 13) {$a^1$}; 
\draw[thick, dashed] (12, 0) -- (12, 13);
\node[above] at (12, 13) {$a^2$}; 
\node[left] at (0, 4) {$\frac{q}{3}$};
\node[left] at (0, 8) {$\frac{2q}{3}$};
\node[left] at (0, 12) {$q$};
\draw[thick, dashed] (0, 4) -- (13, 4);
\node[right] at (13, 4) {$c^1$}; 
\draw[thick, dashed] (0, 8) -- (13, 8);
\node[right] at (13, 8) {$c^2$}; 
\draw[thick, dashed] (0, 12) -- (13, 12);
\node[right] at (13, 12) {$c^3$};
\draw[thick, dashed] (0, 4) -- (6, 0);
\node[above right] at (3, 2) {$b^1$}; 
\draw[thick, dashed] (0, 8) -- (12, 0);
\node[above right] at (6, 4) {$b^2$}; 
\draw[thick, dashed] (0, 12) -- (13, 10/3);
\node[above right] at (6, 8) {$b^3$}; 
\node at (1.5, 1.5) {\Large $1$};
\node at (4.5, 2.5) {\Large $b$};
\node at (8, 1.5) {\Large $ab$};
\node at (10.5, 2.5) {\Large $ab^2$};
\node at (1.5, 5.5) {\Large $bc$};
\node at (1.5, 9.5) {\Large $b^2c^2$};
\node at (8, 5.5) {\Large $ab^2c$};
\node at (4.5, 6.5) {\Large $b^2c$};
\node at (10.5, 6.5) {\Large $ab^3c$};
\node at (13, 1.5) {\Large $a^2b^2$};
\draw[very thick, blue] (12, 0) -- (0, 12);
\node[blue, above right] at (3, 9) {$I = (a^2b^2, b^3c^3)$};
\draw[very thick, blue] (10, 0) -- (0, 10);
\node[blue, above right] at (3, 7) {$I^{[j_3]} = (ab, b^2c)$};
\draw[very thick, blue] (6, 0) -- (0, 6);
\node[blue, above right] at (3, 3) {$I^{[1/2]} = (b)$};
\end{tikzpicture}
}
\end{center}

To demonstrate the tools from the previous section, we compute the Frobenius powers for a specific example and show how the characteristic of $K$ can affect the results.

\begin{prop}\label{prop:squarecube}
Let $S = \kk[x, y, z]$ and $I = (x^2y^2, y^3z^3)$.  Then:
\begin{enumerate}[label = \textnormal{(\alph*)}]
\item If $p=2$, then:
\[ I^{[t]} = \left\{\begin{array}{cc} 
S, & t \in [0, \frac{1}{2}) \\[1 ex] 
(xy, yz), & t \in [\frac{1}{2}, \frac{3}{4}) \\[1 ex]
(xy, y^2z), & t \in [\frac{3}{4}, 1)
\end{array}\right. \]
\item If $p = 3$, then:
\[ I^{[t]} = \left\{\begin{array}{cc} 
S, & t \in [0, \frac{1}{2}) \\[1 ex] 
(y), & t \in [\frac{1}{2}, \frac{2}{3}) \\[1 ex]
(xy, yz), & t \in [\frac{2}{3}, \frac{5}{6}) \\[1 ex]
(xy, y^2z), & t \in [\frac{5}{6}, 1)
\end{array}\right. \]
\item If $p \equiv 1 \pmod{3}$, then:
\[ I^{[t]} = \left\{\begin{array}{cc} 
S, & t \in [0, \frac{1}{2}) \\[1 ex] 
(y), & t \in [\frac{1}{2}, \frac{5}{6}) \\[1 ex]
(xy, y^2z), & t \in [\frac{5}{6}, 1)
\end{array}\right. \]
\item If $p \equiv 2 \pmod{3}$ and $p \neq 2$, then:
\[ I^{[t]} = \left\{\begin{array}{cc} 
S, & t \in [0, \frac{1}{2}) \\[1 ex] 
(y), & t \in [\frac{1}{2}, \frac{5p-1}{6p}) \\[1 ex]
(xy, y^2z), & t \in [\frac{5p-1}{6p}, 1)
\end{array}\right. \]
\end{enumerate}
\end{prop}

\begin{proof}
		We compute $\lambda_{(0,1,0)}$; the other computations are similar.  A (family of) witnesses will be a collection of pairs $\mathbf{v}_{e}=(a_{e},b_{e})$ satisfying
		\[
			\begin{pmatrix}2&0\\2&3\\0&3\end{pmatrix}\begin{pmatrix}a_{e}\\b_{e}\end{pmatrix}<\begin{pmatrix}1\\2\\1\end{pmatrix}.
		\]				
		That is, $a<\frac{1}{2}$ and  $b<\frac{1}{3}$ such that $a+b$ adds without carries in base $p$.

	If $p=2$, the binary representations of $\frac{1}{2}$ and $\frac{1}{3}$ are $.0\overline{1}$ and $.\overline{01}$, respectively.  Thus $a_{e}$ and $b_{e}$ must each have a zero in the first decimal place; the best we can do for the sum without a carry is $.0\overline{1}$.  This is realized by, for example, taking $a_{e}=\trunc_{e}(\frac{1}{2}), b_{e}=0$.  Thus $\lambda_{(0,1,0)}=\frac{1}{2}$, i.e., $y\in I^{[\frac{k}{q}]}$ if and only if $\frac{k}{q}<\frac{1}{2}$.
	
	If $p=3$, the trinary representations of $\frac{1}{2}$ and $\frac{1}{3}$ are $.\overline{1}$ and $.0\overline{2}$, respectively. Thus the first digit of $a_{e}$ must be at most $1$ and the first digit of $b_{e}$ must be a zero.  Without carries, the first digit of $a_{e}+b_{e}$ cannot be more than $1$, so the best we can do for the sum is $.1\overline{2}=\frac{2}{3}$.  This is realized by, for example, taking $a_{e}=.1=$, $b_{e}=\trunc_{e}(\frac{1}{3})$.  Thus $\lambda_{(0,1,0)}=\frac{2}{3}$, i.e., $y\in I^{[\frac{k}{q}]}$ if and only if $\frac{k}{q}<\frac{2}{3}$.

	If $p\equiv 1 \pmod{6}$, set $m=p-1$, $s=\frac{m}{2}$, and $m=\frac{q}{3}$.  The base $p$ representations of $\frac{1}{2}$ and $\frac{1}{3}$ are $.\overline{s}$ and $.\overline{t}$.  Since $s+t<m$, we may add these without carries.  The (unique) family of witnesses is $a_{e}=\trunc_{e}(\frac{1}{2}), b_{e}=\trunc_{e}(\frac{1}{3})$.  We conclude that $\lambda_{(0,1,0)}=\frac{1}{2}+\frac{1}{3}=\frac{5}{6}$, i.e., $y\in I^{[\frac{k}{q}]}$ if and only if $\frac{k}{q}<\frac{5}{6}$.

	If $p\equiv 5 \pmod{6}$, set $m=p-1$, $r=p-2$, $s=\frac{m}{2}$, $t=\frac{r}{3}$, and $w=2t+1$.   The base $p$ representations of $\frac{1}{2}$ and $\frac{1}{3}$ are $.\overline{s}$ and $.\overline{tw}$.
	In order to add without carries, the first digit of $a_{e}+b_{e}$ must be at most $(s+t)$.  But $s+w>p$, so the second digit and all subsequent digits can be arbitrary.  The best we can hope for when adding without carries is $.g\overline{m}$, where $g=s+t$.   One witness is $a_{e}=\trunc_{e}(.\overline{s}), b_{e}=\trunc_{e}(.t\overline{s})$.  Thus $\lambda_{(0,1,0)}=.g\overline{m}=\frac{g+1}{p}=\frac{5p-1}{6p}$, i.e., $y\in I^{[\frac{k}{q}]}$ if and only if $\frac{k}{q}<\frac{5p-1}{6p}$.
\end{proof}

\begin{prop}\label{prop:heightone}
Let $I = (m_1, m_2, \ldots, m_t) \subseteq S = k[x_1, x_2, \ldots, x_n]$ be a monomial ideal of height $1$. Without loss of generality, say $x_1 | m_i$ for all $i$. For all $i$, let $d_i$ denote the power of $x_1$ in $m_i$, and suppose that $d=d_1 < d_i$ for all $i >1$. Suppose also that $m_1=m^{d}$ for some squarefree monomial $m$. Then $\lce(I) = 1/d$. 
\end{prop}

\begin{proof}
We apply Algorithm \ref{algo:first} to compute $\lce(I)$. Note that the first row of the matrix $A$ is 
\[
\left(\begin{array}{cccc}d & d_2 & \cdots & d_t\end{array}\right),
\]
and the first column of $A$ is composed of entries that are either $0$ or $d$. Since we are computing $\lce(I)$, we set $\vec{b}$ equal to the zero vector. By our observations about the matrix $A$, the vector $\vec{v}$ that maximizes $\norm{\vec{v}}$ and satisfies $A \vec{v} \prec p\ones$ is of the form 
\[
\vec{v}_1 = \left(\begin{array}{c}a_1 \\0 \\\vdots \\0\end{array}\right),
\]
where $a_1$ is the largest integer such that $a_1d < p$ ($a$ could be zero). So we add $\frac{1}{p}\vec{v}_1$ to $\mathcal{L}$. We then compute the remainder $\vec{r}$ via the algorithm, and again we notice that the $\vec{v}$ maximizing $\norm{\vec{v}}$ with $A\vec{v} \prec \vec{r}$ is of the form 
\[
\vec{v}_2 = \left(\begin{array}{c}a_2 \\0 \\\vdots \\0\end{array}\right).
\]
Indeed, working through the algorithm, we see that $a_2$ is the largest integer such that 
\[
a_2 d < p^2 - pa_1d, 
\]
and so $\mathcal{L}$ now contains $\frac{1}{p}\vec{v}_1 + \frac{1}{p^2}\vec{v}_2$, and $\norm{\frac{1}{p}\vec{v}_1 + \frac{1}{p^2}\vec{v}_2} = \frac{a_1}{p} + \frac{a_2}{p^2}$. We perform this algorithm infinitely many times, obtaining
\[
\max \{\norm{\vec{v}}: \vec{v}\in \mathcal{L}\} = \sum_{i = 1}^\infty \frac{a_i}{p^i}, 
\]
where for each $i$, $a_i$ is the largest integer such that 
\[
a_i d < p^i - p^{i-1}a_{i-1}d - p^{i-2}a_{i-2}d - \cdots - pa_1d.
\]
Thus, this sum is simply the infinite expansion of $1/d$. 
\end{proof}

We note that we could have also proven the above by observing that the vector 
\[
\vec{v} = \left(\begin{array}{c}1/d \\0 \\\vdots \\0\end{array}\right)
\]
maximizes $\norm{\vec{v}}$ among all $\vec{v} \in \mathbb{R}^n$ satisfying $A\vec{v} \preceq \ones$. Setting 
\[
\vec{v}_e =  \left(\begin{array}{c}\trunc_e(1/d) \\0 \\\vdots \\0\end{array}\right)
\]
would then produce a convergent family $\{ \vec{v}_e\}$ of witnesses, with $\lce(I) = \lim_{e \rightarrow \infty} \norm{\vec{v}_e} = 1/d$.

By Corollary~\ref{cor:sqfree}, every monomial ideal $I$ containing a squarefree monomial satisfies $I^{[t]}=(1)$. More generally, we can ask the following questions.

\begin{question}\label{q:samepowers}
Under what circumstances do two monomial ideals $I$ and $J$ satisfy $I^{[t]} = J^{[t]}$ for all $t < 1$?
\end{question}

\begin{question}\label{q:lce1}
Can we characterize which monomial ideals $I$ have $\lce(I) = 1$?
\end{question}

We note that Algorithm \ref{algo:first} also proves the converse to Corollary \ref{cor:sqfree} in the case $p = 2$. Indeed, if $I$ contains no squarefree monomial, then the vector $\vec{v} \in \mathbb{N}^n$ maximizing $\norm{\vec{v}}$ with $A \vec{v} \prec 2 \cdot \ones$ is easily seen to be the zero vector, meaning that $\lce(I) < 1$ in this case.

\bigskip

\noindent \textbf{Acknowledgments.} We thank Daniel Hern\'andez, Nishad Mandlik, and Emily Witt for helpful conversations. The work in this paper was partially supported by grant \#422465 from the Simons Foundation to the first author.


\end{spacing}


\section*{Appendix: Some Examples}
\renewcommand{\thesection}{A}

\paragraph{} Here we demonstrate Algorithm \ref{algo:first} in two small cases. Throughout, we use the ideal $I = (x^2y^2, y^3z^3) \subseteq S = \kk[x, y, z]$.  Thus, the corresponding matrix $A$ of exponent vectors is
\[ A = \left(\begin{array}{cc}2 & 0 \\2 & 3 \\0 & 3\end{array}\right). \]

\subsection{First Example} 

\paragraph{} Here, we compute the least critical exponent in the case $p = 3$. That is, we compute $\lambda_{\vec{b}}$ where $\vec{b} = (0, 0, 0)$. \\

We start by initializing $\vec{u} = \left(\begin{array}{c}0 \\0\end{array}\right)$. We now enumerate the steps in Algorithm \ref{algo:first}:

\begin{enumerate}
\item Set $\vec{r} = \vec{b} + \ones - A\vec{u} = \vec{b} + \ones =\left(\begin{array}{c}1 \\1 \\1\end{array}\right) $.

\item Find all solutions to $\left(\begin{array}{cc}2 & 0 \\2 & 3 \\0 & 3\end{array}\right)\vec{v} \prec p^1\vec{r} = 3 \left(\begin{array}{c}1 \\1 \\1\end{array}\right) = \left(\begin{array}{c}3 \\3 \\3\end{array}\right) $ satisfying $\norm{\vec{v}} < 3$ and maximizing $\norm{\vec{v}}$. Clearly there is only one such vector, namely $\vec{v} = \left(\begin{array}{c}1 \\0\end{array}\right)$. (The vector $\left(\begin{array}{c}0 \\0\end{array}\right)$ satisfies the inequality, but does not maximize $\norm{\vec{v}}$.) 

\item Now add $\vec{u} + \frac{1}{3}\vec{v} = \left(\begin{array}{c}0 \\0\end{array}\right) + \frac{1}{3}\left(\begin{array}{c}1 \\0\end{array}\right) = \left(\begin{array}{c} 1/3 \\0\end{array}\right)= $ to the list $\mathcal{L}$ of candidates.
\end{enumerate}

We now repeat the algorithm with $\vec{u} = \left(\begin{array}{c} 1/3 \\0\end{array}\right)$.

\begin{enumerate}
\item Set $\vec{r} = \vec{b} + \ones - A\vec{u} = \left(\begin{array}{c}1 \\ 1 \\ 1\end{array}\right) - \left(\begin{array}{c}2/3 \\ 2/3 \\ 0\end{array}\right) = \left(\begin{array}{c} 1/3 \\ 1/3 \\ 1\end{array}\right)$.

\item Find all solutions to $\left(\begin{array}{cc}2 & 0 \\2 & 3 \\0 & 3\end{array}\right)\vec{v} \prec p^2 \vec{r} = 9 \left(\begin{array}{c}1/3 \\ 1/3 \\1\end{array}\right) = \left(\begin{array}{c}3 \\3 \\9\end{array}\right) $ satisfying $\norm{\vec{v}} < 3$ and maximizing $\norm{\vec{v}}$. Again, there is just one such vector, namely $\vec{v} = \left(\begin{array}{c} 1 \\0\end{array}\right)$.

\item We add $\vec{u} + \frac{1}{9}\vec{v} = \left(\begin{array}{c} 1/3 \\0\end{array}\right) + \frac{1}{9}\left(\begin{array}{c} 1 \\0\end{array}\right) = \left(\begin{array}{c} 1/3 + 1/9 \\0\end{array}\right)$ to $\mathcal{L}$.
\end{enumerate}

Continuing this algorithm, we see that the vector $\vec{v}$ found in Step 2 will always be $\left(\begin{array}{c} 1 \\0\end{array}\right)$, and thus at the $e^{\text{th}}$ step, the following vector will be added to $\mathcal{L}$:
\[ \vec{u}_e=\left(\begin{array}{c} 1/3 + 1/9 + 1/27 + \cdots + 1/3^e\\0\end{array}\right). \]

Therefore $\lambda_\vec{b} =\lim_{e \to \infty} \norm{\vec{u}_e} = \frac{1}{2}$.

\subsection{Second Example}

\paragraph{} We now demonstrate a more involved example. In the first example, our work was made easier by the fact that there was only one $\vec{v}$ found in each iteration of Step 2. In general, however, there will likely be several such vectors found, and the magnitudes of the resulting vectors in $\mathcal{L}$ must be compared. \\

For this example, we stay with the same ideal $I$, set $p = 5$, and compute $\lambda_\vec{b}$ for $\vec{b} = (1, 1, 0)$. This will be the smallest Frobenius power of $I$ excluding the monomial $xy$. \\

As before, we initialize $\vec{u} = \left(\begin{array}{c} 0 \\0\end{array}\right)$.

\begin{enumerate}
\item Set $\vec{r} = \vec{b} + \ones - A\vec{u} = \left(\begin{array}{c} 1 \\ 1 \\0\end{array}\right) + \left(\begin{array}{c} 1 \\ 1 \\ 1 \end{array}\right) = \left(\begin{array}{c} 2 \\ 2 \\1 \end{array}\right)$.

\item Find all solutions to $ \left(\begin{array}{cc}2 & 0 \\2 & 3 \\0 & 3\end{array}\right)\vec{v} \prec p^1 \vec{r} = 5 \left(\begin{array}{c} 2 \\ 2 \\1\end{array}\right) = \left(\begin{array}{c}10 \\ 10 \\ 5\end{array}\right) $ satisfying $\norm{\vec{v}} < 5$ and maximizing $\norm{\vec{v}}$. Note that, here there are two possible solutions: $\vec{v}_1 = \left(\begin{array}{c} 3 \\ 1 \end{array}\right)$ and $\vec{v}_2 = \left(\begin{array}{c} 4 \\0\end{array}\right)$. 

\item We then add both $\vec{u}_1= \vec{u} + \frac{1}{5} \vec{v}_1 = \left(\begin{array}{c} 3/5 \\ 1/5\end{array}\right)$ and $\vec{u}_2= \vec{u} + \frac{1}{5} \vec{v}_2 = \left(\begin{array}{c} 4/5 \\ 0\end{array}\right)$ to the set $\mathcal{L}$ of candidates.
\end{enumerate}

The key point is that we need to perform the algorithm for each of the $e$-witnesses in $\mathcal{L}$.  \\

We first perform the algorithm for $\vec{u}_1 = \left(\begin{array}{c} 3/5 \\ 1/5\end{array}\right)$. 

\begin{enumerate}
\item Set $\vec{r} = \vec{b} + \ones - A\vec{u}_1 = \left(\begin{array}{c} 2 \\ 2 \\ 1 \end{array}\right) - \left(\begin{array}{c} 6/5 \\ 9/5 \\ 3/5 \end{array}\right) = \left(\begin{array}{c} 4/5 \\  1/5 \\ 2/5\end{array}\right)$.

\item Find all solutions to $ \left(\begin{array}{cc}2 & 0 \\2 & 3 \\0 & 3\end{array}\right)\vec{v} \prec 5^2 \left(\begin{array}{c} 4/5 \\  1/5 \\ 2/5\end{array}\right) = \left(\begin{array}{c} 20 \\  5 \\ 10\end{array}\right)$ satisfying $\norm{\vec{v}} < 5$ and maximizing $\norm{\vec{v}}$. The $\vec{v}$ we obtain is $ \left(\begin{array}{c} 2 \\ 0 \end{array}\right)$. 

\item Add the vector $\vec{u}_1 + \frac{1}{5^2}\vec{v} =  \left(\begin{array}{c} 3/5 \\ 1/5\end{array}\right) +  \left(\begin{array}{c} 2/25 \\  0\end{array}\right) =  \left(\begin{array}{c} 3/5 \\ 1/5\end{array}\right) =  \left(\begin{array}{c} 17/25 \\ 1/5\end{array}\right)$ to $\mathcal{L}$. Call this vector $\vec{x}$. 
\end{enumerate}

We describe the algorithm a little more briefly for $\vec{u}_2$. 

\begin{enumerate}
\item Set $\vec{r} =  \left(\begin{array}{c} 2/5 \\ 2/5 \\ 0\end{array}\right)$. 

\item Obtain the vector $\vec{v} =  \left(\begin{array}{c} 4 \\ 0 \end{array}\right)$.

\item Add $\vec{u}_2 + \frac{1}{5^2} \left(\begin{array}{c} 4 \\ 0\end{array}\right) =  \left(\begin{array}{c} 4/5 \\ 0\end{array}\right) +  \left(\begin{array}{c} 4/25 \\ 0\end{array}\right) =  \left(\begin{array}{c} 24/25 \\ 0\end{array}\right)$ to $\mathcal{L}$. Call this vector $\vec{y}$. 
\end{enumerate}

In the first case, $\norm{\vec{x}} = 22/25$. In the second case, $\norm{\vec{y}} = 24/25$. Indeed, if we continue running the algorithm as in the second case, we will continue to obtain the vector $\vec{v} =  \left(\begin{array}{c} 4 \\ 0\end{array}\right)$ in Step 2. At the $e^{\text{th}}$ stage, $\mathcal{L}$ will contain the vector: 
\[ \vec{u}_e = \left(\begin{array}{c} 4/5 + 4/25 + 4/125 + \cdots + 4/5^{e} \\ 0\end{array}\right) \]
Therefore $\lambda_\vec{b} = \lim_{e \to \infty} \norm{\vec{u}_e} = 1$.


\begin{thebibliography}{{\c C}WW95a}

\bibitem[BMS08]{eth:roots}
M.~Blickle, M.~Musta\c{t}\v{a}, and K.~Smith.
\newblock Discreteness and rationality of F-thresholds. 
\newblock Special volume in honor of Melvin Hochster. 
\newblock {\em Michigan Math. J.} 57 (2008), 43--61. 

\bibitem[Fed83]{Fedder}
R.~Fedder.
\newblock $F$-purity and rational singularity. 
\newblock {\em Trans. Amer. Math. Soc.} 278 (1983), no. 2, 461--480.

\bibitem[M2]{Macaulay2}
D.~Grayson and M.~Stillman.
\newblock Macaulay2, a software system for research in algebraic geometry. 
\newblock Available at \url{http://www.math.uiuc.edu/Macaulay2/}.

\bibitem[Her12]{F-purity:of:hypersurfaces}
D.~Hern\'andez.
\newblock F-purity of hypersurfaces.
\newblock {\em Math. Res. Lett.} 19 (2012), no. 2, 389--401.

\bibitem[HTW18]{Frobenius:powers}
D.~Hern\'andez, P.~Teixeira, and E.~Witt.
\newblock Frobenius powers.	
\newblock \href{https://arxiv.org/abs/1802.02705v1}{arXiv:1802.02705 [math.AC]}

\bibitem[HTW20]{Frobenius:powers:of:some:monomial:ideals}
D.~Hern\'andez, P.~Teixeira, and E.~Witt.
\newblock Frobenius powers of some monomial ideals.	
\newblock {\em J. Pure Appl. Algebra} 224 (2020), no. 1, 66--85.

\bibitem[How01]{lct:of:monomial:ideals}
J.~Howald.
\newblock Multiplier ideals of monomial ideals.
\newblock {\em Trans. Amer. Math. Soc.} 353 (2001), no. 7, 2665--2671. 

\bibitem[How03]{suffgen}
J.~Howald.
\newblock Multiplier ideals of sufficiently general polynomials.
\newblock \href{https://arxiv.org/abs/math/0303203}{arXiv:0303203 [math.AG]}

\end{thebibliography}
\end{document}